\documentclass[reqno]{amsart}

\usepackage{amscd, amssymb,amsfonts}
\usepackage{a4wide}
\usepackage{amsmath,amssymb}
\usepackage{mathrsfs}
\usepackage{graphicx}
\usepackage{subcaption}

\newtheorem{defn}{Definition}
\newtheorem{theorem}{Theorem}
\newtheorem{lem}{Lemma}
\newtheorem{prop}{Proposition}

\makeatletter
\newcommand*{\dif}{\@ifnextchar^{\DIfF}{\DIfF^{}}}
\def\DIfF^#1{\mathop{\mathrm{\mathstrut d}}\nolimits^{#1}\gobblesp@ce}
\def\gobblesp@ce{\futurelet\diffarg\opsp@ce}
\def\opsp@ce{%
  \let\DiffSpace\!%
  \ifx\diffarg(%
    \let\DiffSpace\relax
  \else
    \ifx\diffarg[%
      \let\DiffSpace\relax
    \else
      \ifx\diffarg\{%
        \let\DiffSpace\relax
      \fi\fi\fi\DiffSpace}
\makeatother

\title[Large Time Step TVD Schemes]
{Large Time Step TVD Schemes for  \\ Hyperbolic Conservation Laws}

\author[Lindqvist, Aursand, Fl\aa{}tten and Solberg]
{Sofia Lindqvist$^{\text{A}}$, Peder Aursand$^{\text{B}}$, Tore Fl\aa{}tten$^{\text{C,E}}$ and Anders Aase Solberg$^{\text{D}}$}

\thanks{$^\text{A}$Dept.\ of Mathematical Sciences, Norwegian University of Science and Technology (NTNU), NO-7491 Trondheim, Norway.}
\thanks{$^\text{B}$SINTEF Energy Research, P.\ O.\ Box 4761 Sluppen, NO-7465 Trondheim, Norway.}
\thanks{$^\text{C}$SINTEF Materials and Chemistry, P.\ O.\ Box 4760 Sluppen, NO-7465 Trondheim, Norway.}
\thanks{$^\text{D}$Dept.\ of Energy and Process Engineering, Norwegian University of Science and Technology (NTNU), NO-7491 Trondheim, Norway.}
\thanks{Email: lindqvist.sofia@gmail.com, Peder.Aursand@sintef.no, Tore.Flatten@sintef.no, \\ andesolb@stud.ntnu.no}
\thanks{$^\text{E}$Corresponding author}
\date{\today}

\begin{document}

\begin{abstract}

Large time step explicit schemes in the form originally proposed by LeVeque  [{\em Comm. Pure Appl. Math.}, 37 (1984), pp.\ 463--477] have seen a significant revival in recent years. In this paper we consider a general framework of {local} $2k+1$ point schemes containing LeVeque's scheme (denoted as {LTS-Godunov}) as a member. A modified equation analysis allows us to interpret each numerical cell interface coefficient of the framework as a {\em partial numerical viscosity} coefficient. 

We identify the least and most diffusive TVD schemes in this framework. The most diffusive scheme is the $2k+1$-point Lax-Friedrichs scheme (LTS-LxF). The least diffusive scheme is the Large Time Step scheme of LeVeque based on Roe upwinding (LTS-Roe). Herein, we prove a generalization of Harten's lemma: all partial numerical viscosity coefficients of any local unconditionally TVD scheme are bounded by the values of the corresponding coefficients of the LTS-Roe and LTS-LxF schemes.

We discuss the nature of entropy violations associated with the LTS-Roe scheme, in particular we extend the notion of {\em transonic rarefactions} to the LTS framework. We provide explicit inequalities relating the numerical viscosities of LTS-Roe and LTS-Godunov across such generalized transonic rarefactions, and discuss numerical entropy fixes. 

Finally, we propose a one-parameter family of Large Time Step TVD schemes spanning the entire range of the admissible total numerical viscosity. Extensions to nonlinear systems are obtained through the Roe linearization. The 1D Burgers equation and the Euler system are used as numerical illustrations.

\end{abstract}

\maketitle

\subsection*{subject classification}
 {65M08, 35L65, 65Y20}
\subsection*{key words}
{hyperbolic conservation laws, large time step, total variation diminishing}

\section{Introduction}

We are concerned with the numerical solution of hyperbolic conservation laws, i.e. initial value problems in the form
\begin{subequations}\label{eq:scalarcons}
\begin{gather}
  u_t+f(u)_x=0, \quad u\in \mathcal{W}\subseteq \mathbb{R}^N, \\
  u(x,0)=u_0(x)
\end{gather}
\end{subequations}
where the Jacobian $f_u$ has real eigenvalues and a complete set of eigenvectors.

Assuming a discretization where
\begin{subequations}
\begin{gather}
  x=j\Delta x, \\
  t=n\Delta t,
\end{gather}
\end{subequations}
we consider the $(2k+1)$ point finite volume approximation 
\begin{subequations}\label{eq:FV}
\begin{equation}\label{eq:FV1}
  U^{n+1}_j=H(U_{j-k}^n,\ldots,U_{j+k}^n)=U_j^n-\frac{\Delta t}{\Delta x}\left(F_{j+1/2}^n-F_{j-1/2}^n\right),
\end{equation}
where
\begin{equation}\label{eq:FV2}
  F_{j+1/2}=F(U_{j-k+1}^n,\ldots,U_{j+k}^n).  
\end{equation}
\end{subequations}
The numerical flux function $F$ is assumed to be Lipschitz continuous and consistent in the sense that
\begin{equation}\label{eq:consistent}
  F(U,\ldots,U)=f(U).  
\end{equation}
If we assume compact support or periodic boundary conditions, weak solutions to the {\em scalar} initial value problem \eqref{eq:scalarcons} possess the following monotonicity properties~\cite{har83}:
\begin{enumerate}
  \item No new local extrema in $x$ may be created;
  \item The value of a local minimum is nondecreasing, the value of a local maximum is nonincreasing.
\end{enumerate}
As a consequence, we have that
\begin{equation}\label{eq:contTVD}
  \frac{\dif}{\dif t}\int\left|\frac{\partial u}{\partial x}\right|\dif x\leq 0.  \end{equation}

In 1983, Harten~\cite{har83} introduced the now classical concept of TVD (total-variation-diminishing) numerical schemes; i.e.\ schemes that respect \eqref{eq:contTVD} on the discrete level. The discrete {\em Total Variation} (TV) at time step $n$ is defined as
\begin{subequations}\label{eq:defTVD}
\begin{equation}
  \mathrm{TV}^n=\sum_j\left|U_{j+1}^n-U_j^n\right|.  
\end{equation}
A numerical scheme for \eqref{eq:scalarcons} is said to be {\em Total Variation Diminishing} (TVD) if
\begin{equation}
  \mathrm{TV}^{n+1}\leq\mathrm{TV}^n.  
\end{equation}
\end{subequations}
The introduction of this strong nonlinear stability condition spawned an intensive research activity in the eighties. In particular, the TVD framework provided a convenient setting for generalizing the MUSCL approach of van Leer~\cite{lee79} for obtaining {\em high resolution} schemes, i.e.\ schemes that achieve second-order accuracy in smooth regions while preserving the robustness of first-order schemes near extrema and discontinuities~\cite{har83}.

It follows from the Lax-Wendroff theorem~\cite{lax60} that TVD schemes in the form \eqref{eq:FV} will converge to a weak solution, but not necessarily the correct entropy solution~\cite{har76,tad84}, of the scalar initial value problem \eqref{eq:scalarcons}.

\subsection{Explicit vs implicit schemes}

Explicit schemes in the form \eqref{eq:FV} suffer from the well-known restriction on the time step denoted as the {\em Courant-Friedrichs-Lewy} (CFL) condition~\cite{cou67}. In particular, a {\em necessary} condition for stability is
\begin{equation}\label{eq:genCFL}
  \bar{C}=\frac{\Delta t}{\Delta x}\left|\max\limits_i\lambda_i\right|\leq k,  
\end{equation}
where the dimensionless time parameter $\bar{C}$ is called the {\em Courant number} and $\lambda_i$ are the eigenvalues of $f_u$.

There is an intricate relationship between {\em stability} and {\em accuracy} of numerical methods for hyperbolic conservation laws. Nonlinear equations will develop discontinuities in finite time, causing oscillations for high-order methods due to the Gibbs phenomenon. In this respect, most studies on $2k+1$ point schemes have focused on employing extended stencils to improve {accuracy} rather than extending {stability}~\cite{har83,har97,jia96}, and typically the CFL condition will reduce to
\begin{equation}\label{eq:strictCFL}
  \bar{C}\leq 1
\end{equation}
regardless of the size of the stencil. Classically, the strategy of choice to overcome \eqref{eq:strictCFL} has been {\em implicit schemes}~\cite{har84}, i.e. methods where the numerical flux \eqref{eq:FV2} is evaluated at time level $n+1$. In particular, if \eqref{eq:FV2} is modified as
\begin{equation}\label{eq:impf}
  F_{j+1/2}=F(U_{j-k+1}^{n+1},\ldots,U_{j+k}^{n+1}),
\end{equation}
then \eqref{eq:FV} is said to be an {\em implicit} scheme.

Harten~\cite{har84} provided TVD conditions for hybrid explicit-implicit schemes. In particular, implicit schemes can be TVD without any restriction on the time step, at the cost of excessive numerical diffusion.

\subsection{Large Time Step explicit schemes}

The observations
\begin{itemize}
  \item explicit schemes are commonly limited by the strict time step restriction~\eqref{eq:strictCFL};
  \item implicit schemes suffer from excessive numerical diffusion;
\end{itemize}
motivate the study of multi-point schemes in the form \eqref{eq:FV}, with the aim of obtaining the relaxed time step restriction \eqref{eq:genCFL}. For the purposes of this paper, we will refer to such schemes as {\em Large Time Step} (LTS) schemes.

Theoretical foundations for this approach were established several decades ago~\cite{har76,jam86,jam87,osh86}. The first investigations of concrete Large Time Step explicit schemes were performed in a series of papers by LeVeque~\cite{lev82,lev84,lev85}. Herein, he proposed a generalization of the classical Godunov scheme based on treating each wave interaction as linear. This procedure leads to a consistent and conservative method satisfying the TVD property~\cite{lev84}. Extensions to systems were considered in~\cite{lev85}. LeVeque observed that his scheme would often be highly accurate, but could suffer from artificial staircase solutions and lack of pointwise convergence for systems~\cite{lev85}.

In~\cite{har86}, Harten proposed a more diffusive (and robust) Large Time Step TVD scheme. For linear equations, his scheme reduces to a repeated application of the standard upwind scheme. Through numerical examples, Harten demonstrated how his modified flux approach (introduced in~\cite{har83}) could be extended as a general method to achieve high resolution TVD LTS schemes.

Nevertheless, the TVD condition is an excessively strict stability criterion. In general, better accuracy can be achieved from the less restrictive ENO~\cite{har97} and WENO~\cite{jia96} schemes which allow spurious oscillations on the level of the truncation error. Goodman and LeVeque~\cite{goo85} also showed that in 2 space dimensions, any TVD scheme is generally at most first-order accurate. Consequently, there was a decline of interest in schemes formulated in the TVD setting, including the LTS schemes.

However, even with the present significant advances in computing power, there is still a need to sacrifice accuracy for efficiency in many practical applications. For instance, simulators for petroleum processing facilities often involve the solution of highly complicated multiphase flow models through several kilometer long pipelines~\cite{dan11}. These simulators are typically based on implicit methods, but we may note that explicit methods in the form \eqref{eq:FV} are trivially parallelizable whereas implicit methods are not~\cite{nor11}. The current trend of moving towards massively parallel computer hardware then provides motivation for revisiting the Large Time Step class of schemes, and indeed these methods have seen a {\em significant revival} in recent years.

Murillo, Morales-Hern\'{a}ndez and co-workers~\cite{mor12a,mor12,mor14,mor13,mur06} have investigated variations of LeVeque's LTS scheme, based on Roe upwinding, for the 1D and 2D shallow water equations. Herein, an emphasis was made on the proper treatment of source terms and boundary conditions.

Qian and Lee~\cite{qia11} successfully extended LeVeque's LTS-Godunov scheme to the full 3D Euler equations through dimensional splitting. In~\cite{qia12}, these authors extended their investigation to some high-resolution LTS schemes constructed through Harten's~\cite{har86} approach. 

Xu et al.~\cite{xu14} applied the LTS-Godunov scheme to the 1D shallow water equations. Recently, Thompson and Moeller~\cite{tho15} independently rediscovered the LTS-Roe algorithm and applied it to Maxwell's equations.

All these authors follow LeVeque~\cite{lev82,lev84,lev85} in conceptually formulating their schemes as {\em wave propagation methods}, and they all observe the following drawbacks:
\begin{itemize}
  \item[D1:] {\em For systems, oscillations tend to occur for large time steps due to the nonlinear wave interactions}.
  \item[D2:] {\em Propagating rarefactions as expansion shocks introduces entropy-violating solutions}.
\end{itemize}
For convex scalar equations, the rigorous LTS-Godunov scheme proposed by LeVeque~\cite{lev84} was proved to converge to the correct entropy solution in~\cite{wan04}. This scheme involves explicitly tracking the spreading of each rarefaction wave across several computational cells, whereas the wave propagation algorithm is most naturally formulated in terms of isolated discontinuities. For this reason, most authors address the issue D2 by resolving rarefactions into sufficiently many individual expansion shocks~\cite{lev82,mor12,qia11,xu14}.

For the classical essentially 3-point explicit schemes, we note that both spurious oscillations and entropy violations are typically handled by applying a controlled amount of numerical diffusion~\cite{evj03,tad84}. For Large Time Step schemes, Xu et al.~\cite{xu14} remark that the special updating procedure of wave propagation methods prohibits a natural introduction of numerical viscosity. This remark motivates our current paper.

\subsection{Outline of the paper}

The primary aim of our paper is to describe a general LTS framework where TVD schemes are classified according to their numerical viscosity, and herein explicitly place the common wave propagation schemes of LeVeque~\cite{lev82,lev84,lev85}. By this we hope to lay a foundation for further systematic development and improvement of explicit LTS schemes.

Our paper is organized as follows. In Section~\ref{sec:multi}, we introduce the concept of {\em local} multi-point schemes and provide the transformation between the viscosity formulation and the flux-difference-splitting formulation of such schemes. Through a modified equation analysis, we obtain a useful expression for the total numerical viscosity inherent in the schemes. We present closed form expressions for the numerical viscosity coefficients of the LTS-Godunov scheme.

In Section~\ref{sec:TVD}, we investigate the TVD conditions for  LTS schemes. In particular, we identify the least and most diffusive of the TVD local schemes. The least diffusive scheme is precisely the LTS scheme of LeVeque where Roe upwinding is used in place of the rigorous Godunov solver. The most diffusive scheme can be viewed as a natural extension to Large Time Steps of the classical Lax-Friedrichs scheme. We discuss the precise conditions under which LTS-Godunov contains more numerical viscosity than LTS-Roe, and how this difference can lead to entropy violations for the LTS-Roe scheme.

In Section~\ref{sec:systems}, we propose a natural one-parameter family of local LTS-TVD schemes that spans the entire admissible range of numerical viscosity. We extend our framework to general nonlinear systems of conservation laws through the Roe linearization.

In Section~\ref{sec:simulations}, we illustrate our concepts through numerical simulations of the Burgers equation and the Euler system. Our main findings are summarized in Section~\ref{sec:summary}.

\section{Multi-point methods}\label{sec:multi}

For scalar equations, standard 3-point schemes in the form \eqref{eq:FV} can be parametrized by a single numerical viscosity coefficient $Q_{j+1/2}^n$ through the formulation~\cite{har83,tad84}:
\begin{equation}\label{eq:3visc}
  F_{j+1/2}=\frac{1}{2}\left(f(U_j^n)+f(U_{j+1}^n)\right)-\frac{1}{2}\frac{\Delta x}{\Delta t}Q_{j+1/2}^n\left(U_{j+1}^n-U_j^n\right),
\end{equation}
which we will denote as the {\em viscosity formulation}. Alternatively, me may split the flux difference in \eqref{eq:FV1} into left-going and right-going fluctuations as
\begin{equation}\label{eq:fluct}
  F_{j+1/2}-F_{j-1/2}=\mathcal{A}^+_{j-1/2}\left(U_j^n-U_{j-1}^n\right)+\mathcal{A}_{j+1/2}^-\left(U_{j+1}^n-U_j^n\right),
\end{equation}
commonly denoted as the {\em flux-difference splitting} formulation~\cite{lev02}. The natural $2k+1$ point extensions of these concepts may be formulated as follows, where for convenience we drop the temporal index $n$. Instead we use upper indices as relative cell interface indicators.
\begin{itemize}
  \item {\em Viscosity formulation}:
\begin{multline}\label{eq:viscform}
  F_{j+1/2}=\frac{1}{2}\left(f_j+f_{j+1}\right)-\frac{1}{2}\frac{\Delta x}{\Delta t}Q_{j+1/2}^0\Delta_{j+1/2} \\ -\frac{\Delta x}{\Delta t}\sum_{i=1}^{\infty}\left(Q^{i-}_{j+1/2-i}\Delta_{j+1/2-i}+Q^{i+}_{j+1/2+i}\Delta_{j+1/2+i}\right).
\end{multline}
  \item {\em Flux-difference splitting (FDS) formulation:}~\cite{jam86}
\begin{equation}\label{eq:fdsform}
  U_j^{n+1}=U_j^n-\frac{\Delta t}{\Delta x}\sum_{i=0}^{\infty}\left(\mathcal{A}^{i+}_{j-1/2-i}\Delta_{j-1/2-i}+\mathcal{A}^{i-}_{j+1/2+i}\Delta_{j+1/2+i}\right).
\end{equation}
\end{itemize}
Herein, we have introduced the shorthands
\begin{subequations}
\begin{gather}
  f_j=f(U_j), \\
  \Delta_{j+1/2}=U_{j+1}-U_j
\end{gather}
\end{subequations}
and the convention
\begin{equation}\label{eq:zeroconv}
  Q^{i\pm}=\mathcal{A}^{i\pm}=0\qquad\text{for}\qquad i\geq k.
\end{equation}
In this paper, we will have a particular focus on {\em local} multi-point schemes which we define as follows.
\begin{defn}\label{def:local}
In the formulation \eqref{eq:viscform}, assume that all numerical viscosity coefficients are bounded and satisfy
\begin{equation}
  Q^{i\pm}_{j+1/2}=Q^{i\pm}(U_j,U_{j+1})\quad\forall i\in\mathbb{N}^0,j\in\mathbb{Z}.
\end{equation}
Likewise, in the formulation \eqref{eq:fdsform} assume that all $\mathcal{A}$ are bounded and satisfy
\begin{equation}
  \mathcal{A}^{i\pm}_{j+1/2}=\mathcal{A}^{i\pm}(U_j,U_{j+1})\quad\forall i\in\mathbb{N}^0,j\in\mathbb{Z}.
\end{equation}
Such schemes will be termed {\bf local} schemes.
\end{defn}

\begin{prop}
For scalar equations, the coefficients $\mathcal{A}$ of local flux-difference splitting schemes~\eqref{eq:fdsform} are uniquely determined.
\end{prop}
\begin{proof}
Assume the initial condition
\begin{subequations}\label{eq:initjump}
\begin{equation}
  U_j^n=\begin{cases} U_{\mathrm{L}} &\text{for}\quad j \leq p \\ U_{\mathrm{R}} &\text{for}\quad j > p \end{cases},
\end{equation}
i.e.
\begin{equation}
  \Delta_{j+1/2+i}^n=\begin{cases} U_{\mathrm{R}}-U_{\mathrm{L}}\neq 0 &\text{for}\quad i+j=p \\ 0 & \text{otherwise.} \end{cases}
\end{equation}
\end{subequations}
Then, by the boundedness of $\mathcal{A}$, we obtain from \eqref{eq:FV} and \eqref{eq:fdsform}:
\begin{equation}
 \frac{\Delta x}{\Delta t}\frac{H(U_{j-k}^n,\ldots,U_{j+k}^n)-U_j^n}{U_{\mathrm{R}}-U_{\mathrm{L}}}=\begin{cases} \mathcal{A}^{i-}(U_{\mathrm{L}},U_{\mathrm{R}}) & \text{for}\quad i=p-j\geq 0 \\ \mathcal{A}^{i+}(U_{\mathrm{L}},U_{\mathrm{R}}) & \text{for}\quad i=j-1-p \geq 0\end{cases}.
\end{equation}
\end{proof}
\begin{prop}
For scalar equations, the coefficients $Q$ of local viscosity schemes~\eqref{eq:viscform} are uniquely determined.
\end{prop}
\begin{proof}
  Consider the initial condition \eqref{eq:initjump}. Then, by the boundedness of $Q$, we obtain from \eqref{eq:FV} and \eqref{eq:viscform}:
\begin{gather}\label{eq:viscuniq}
 \frac{H(U_{j-k}^n,\ldots,U_{j+k}^n)-U_j^n}{U_{\mathrm{R}}-U_{\mathrm{L}}}=\begin{cases} \frac{1}{2}\left(C+Q^0+2Q^{1-}\right)(U_{\mathrm{L}},U_{\mathrm{R}}) & \text{for}\quad j=p \\  \frac{1}{2}\left(C-Q^0+2Q^{1+}\right)(U_{\mathrm{L}},U_{\mathrm{R}}) & \text{for}\quad j=p+1 \\  \left(Q^{i-}-Q^{(i+1)-}\right)(U_{\mathrm{L}},U_{\mathrm{R}}) & \text{for}\quad i=j-p-1\geq 1 \\  \left(Q^{(i+1)+}-Q^{i+}\right)(U_{\mathrm{L}},U_{\mathrm{R}}) & \text{for}\quad i=p-j\geq 1,\end{cases}
\end{gather}
where $C$ is the signed local cell interface Courant number:
\begin{equation}\label{eq:defcou}
  C(U_{\mathrm{L}},U_{\mathrm{R}})=\begin{cases} \frac{\Delta t}{\Delta x}f'(u) & \mathrm{if}\quad U_{\mathrm{L}}=U_{\mathrm{R}}=u, \\ \frac{\Delta t}{\Delta x}\frac{f(U_{\mathrm{R}})-f(U_{\mathrm{L}})}{U_{\mathrm{R}}-U_{\mathrm{L}}}  & \mathrm{otherwise}. \end{cases}
\end{equation}
Now $Q^{k\pm}(U_{\mathrm{L}},U_{\mathrm{R}})=0$ according to \eqref{eq:zeroconv}. Hence \eqref{eq:viscuniq} uniquely determines $Q^{(k-1)\pm}(U_{\mathrm{L}},U_{\mathrm{R}})$. Then, the remaining $Q(U_{\mathrm{L}},U_{\mathrm{R}})$ are uniquely determined by recursion.
\end{proof}

\begin{prop}
  For a given local multi-point scheme, there is a one-to-one mapping between the coefficients $\mathcal{A}$ of \eqref{eq:fdsform} and the coefficients $Q$ of \eqref{eq:viscform} as follows:
\begin{subequations}\label{eq:Q2A}
\begin{align}
  \mathcal{A}^{0+}&=\frac{1}{2}\frac{\Delta x}{\Delta t}\left(C+Q^0-2Q^{1-}\right), \\
  \mathcal{A}^{0-}&=\frac{1}{2}\frac{\Delta x}{\Delta t}\left(C-Q^0+2Q^{1+}\right), \\
  \mathcal{A}^{i+}&=\frac{\Delta x}{\Delta t}\left(Q^{i-}-Q^{(i+1)-}\right)\qquad i\in\{1,\ldots,k-1\}, \\
  \mathcal{A}^{i-}&=\frac{\Delta x}{\Delta t}\left(Q^{(i+1)+}-Q^{i+}\right)\qquad i\in\{1,\ldots,k-1\}
\end{align}
\end{subequations}
and
\begin{subequations}\label{eq:A2Q}
\begin{align}
  Q^{i+}&=-\frac{\Delta t}{\Delta x}\sum_{p=i}^{\infty}\mathcal{A}^{p-}, \\
   Q^0&=\frac{\Delta t}{\Delta x}\sum_{p=0}^{\infty}\left(\mathcal{A}^{p+}-\mathcal{A}^{p-}\right), \\
   Q^{i-}&=\frac{\Delta t}{\Delta x}\sum_{p=i}^{\infty}\mathcal{A}^{p+}.
\end{align}
\end{subequations}
\end{prop}
\begin{proof}
Use the definition~\eqref{eq:defcou} together with~\eqref{eq:viscform} to express the flux difference~\eqref{eq:fluct} as
\begin{multline}\label{eq:fluct2}
  \frac{\Delta t}{\Delta x}\left(F_{j+1/2}-F_{j-1/2}\right)=\frac{1}{2}\left(C_{j+1/2}-Q_{j+1/2}^0+2Q_{j+1/2}^{1+}\right)\Delta_{j+1/2} \\ +\frac{1}{2}\left(C_{j-1/2}+Q_{j-1/2}^0-2Q_{j-1/2}^{1-}\right)\Delta_{j-1/2} \\
  -\sum_{i=1}^{k-1}\left(\left(Q^{(i+1)-}_{j-1/2-i}-Q^{i-}_{j-1/2-i}\right)\Delta_{j-1/2-i}+\left(Q^{i+}_{j+1/2+i}-Q^{(i+1)+}_{j+1/2+i}\right)\Delta_{j+1/2+i}\right)
\end{multline}
and equate the coefficients of \eqref{eq:fdsform} and \eqref{eq:fluct2}.
\end{proof}

\subsection{Modified equation}

These multi-point methods generally yield first order accurate approximations to \eqref{eq:scalarcons}. We now wish to quantify the amount of numerical diffusion in LTS schemes by deriving the convection-diffusion equation for which the methods would be {\em second-order} accurate. Such a modified equation was derived by Harten et al.~\cite{har76}:
\begin{lem}
  A scheme in the form \eqref{eq:FV} approximates to second order the equation
\begin{equation}\label{eq:harmod}
  u_t+f(u)_x=\frac{1}{2}\Delta x\left[\frac{\Delta x}{\Delta t}\left(\sum_{\ell=-k}^k\ell^2\frac{\partial H}{\partial U_{\ell}}(u,\ldots,u)-c^2\right)u_x\right]_x,
\end{equation}
where 
\begin{equation}\label{eq:contcou}
  c=\frac{\Delta t}{\Delta x}f'(u).
\end{equation}
\end{lem}
We now wish to express this equation in terms of the coefficients $Q$ of \eqref{eq:viscform}.
\begin{prop}
A $2k+1$ point scheme in the form~\eqref{eq:viscform} will give a second-order accurate approximation to the equation
\begin{equation}\label{eq:modQ}
  u_t+f(u)_x=\frac{1}{2}\Delta x\left[\frac{\Delta x}{\Delta t}\left(\bar{Q}^0-c^2+\sum_{i=1}^{k-1}2\left(\bar{Q}^{i-}+\bar{Q}^{i+}\right)\right)u_x\right]_x,
\end{equation}
where
\begin{equation}
  \bar{Q}=Q(u,\ldots,u).  
\end{equation}
\end{prop}
\begin{proof}
At $(U_{-k},\ldots,U_{k})=(u,\ldots,u)$ we have by the product rule
\begin{subequations}
\begin{gather}
  \frac{\partial H}{\partial U_{\ell}}=\frac{\Delta t}{\Delta x}\left(\mathcal{A}^{(\ell-1)+}-\mathcal{A}^{\ell+}-(u-u)\sum_{i=0}^{k-1}\left(\frac{\partial\mathcal{A}^{i+}}{\partial U_{\ell}}+\frac{\partial\mathcal{A}^{i-}}{\partial U_{\ell}}\right)\right)\quad\text{for}\quad\ell\in\{-k,\ldots,-1\}, \\
  \frac{\partial H}{\partial U_{\ell}}=\frac{\Delta t}{\Delta x}\left(\mathcal{A }^{\ell-}-\mathcal{A}^{(\ell-1)-}-(u-u)\sum_{i=0}^{k-1}\left(\frac{\partial\mathcal{A}^{i+}}{\partial U_{\ell}}+\frac{\partial\mathcal{A}^{i-}}{\partial U_{\ell}}\right)\right)\text{for}\quad\ell\in\{1,\ldots,k\},
\end{gather}
\end{subequations}
which simplifies to
\begin{equation}
  \frac{\partial H}{\partial U_{\pm\ell}}=\pm\frac{\Delta t}{\Delta x}\left(\mathcal{A}^{\ell\mp}-\mathcal{A}^{(\ell-1)\mp}\right)\quad\text{for}\quad \ell\in\{1,\ldots,k\}.
\end{equation}
By the transformation \eqref{eq:Q2A} this can be written as
\begin{subequations}
\begin{gather}
  \frac{\partial H}{\partial U_{\pm\ell}}=\bar{Q}^{(\ell-1)\pm}-2\bar{Q}^{\ell\pm}+\bar{Q}^{(\ell+1)\pm}\quad\text{for}\quad \ell\in\{2,\ldots,k\}, \\
  \frac{\partial H}{\partial U_{\pm 1}}=\frac{1}{2}\left(\bar{Q}^0\mp c\right)-2\bar{Q}^{\pm 1}+\bar{Q}^{\pm 2}.
\end{gather}
\end{subequations}
Substituting into \eqref{eq:harmod} gives the desired result.
\end{proof}
The expression
\begin{equation}\label{eq:defD}
  D(u)=\bar{Q}^0-c^2+\sum_{i=1}^{k-1}2\left(\bar{Q}^{i-}+\bar{Q}^{i+}\right)
\end{equation}
may now be interpreted as a measure of the amount of numerical diffusion inherent in the scheme. Observe now that $D$ is a monotonically increasing function of each $Q$; this motivates denoting the parameters $Q$ as {\em partial numerical viscosity coefficients}.

We will now present closed form expressions for the numerical viscosity coefficients of LeVeque's LTS-Godunov scheme~\cite{lev82,lev84,lev85}.

\subsection{LTS-Godunov}

In the original Godunov's method, $U_j^{n+1}$ is taken to be the average value of the exact solution $u(x,t^{n+1})$ over the cell $(x_{j-1/2},x_{j+1/2})$, i.e.
\begin{equation}\label{eq:godproj}
  U_j^{n+1}=\frac{1}{\Delta x}\int_{x_{j-1/2}}^{x_{j+1/2}}u(x,t^{n+1})\dif x,
\end{equation}
subject to the piecewise constant initial data
\begin{equation}\label{eq:piececonst}
  u(x,t^n)=U_j^n\quad\text{for}\quad x_{j-1/2}\leq x<x_{j+1/2}.
\end{equation}
Herein, we note that each cell interface defines a local {\em Riemann problem:}
\begin{equation}\label{eq:ciriemann}
  \hat{u}_{j+1/2}(x,t^n)=\begin{cases} U_j & \text{for}\quad x < x_{j+1/2} \\ U_{j+1} & \text{for}\quad x \geq x_{j+1/2} \end{cases}.
\end{equation}
For scalar equations, the exact unique entropy solution $\hat{u}_{j+1/2}(x,t)$ to the initial value problem defined by~\eqref{eq:ciriemann} may be written in closed form as follows:
\begin{equation}\label{eq:oshu}
   \hat{u}_{j+1/2}(x,t)=\hat{u}_{j+1/2}(\zeta(x,t))=\begin{cases} -\frac{\dif}{\dif\zeta}\left(\min\limits_{u\in[U_j,U_{j+1}]}[f(u)-\zeta u]\right) & \text{for}\quad U_j<U_{j+1} \\ -\frac{\dif}{\dif\zeta}\left(\max\limits_{u\in[U_{j+1},U_{j}]}[f(u)-\zeta u]\right) & \text{for}\quad U_j\geq U_{j+1} \end{cases}
\end{equation}
where
\begin{equation}
  \zeta=\frac{x-x_{j+1/2}}{t-t^n}.  
\end{equation}
This expression, due to Osher~\cite{osh84}, is valid for general nonlinear flux functions $f(u)$. Then, as long as the solutions of each cell interface Riemann problem do not interact, the numerical flux of the classical Godunov scheme can be written as~\cite{osh84}:
\begin{equation}
  F_{j+1/2}=\begin{cases} \min\limits_{u\in[U_j,U_{j+1}]} f(u) & \text{for}\quad U_j<U_{j+1} \\  \max\limits_{u\in[U_{j+1},U_{j}]} f(u) & \text{for}\quad U_j\geq U_{j+1} \end{cases}.
\end{equation}
However, for Courant numbers greater than 1, wave interactions must be taken into account, and maintaining the exact projection \eqref{eq:godproj} becomes an intractable problem.

In a series of papers, LeVeque~\cite{lev82,lev84,lev85} proposed a simplified Large Time Step generalization of Godunov's method based on treating each wave interaction as linear. Then the numerical solution no longer represents the projected exact solution after each time step, but a conservative and consistent method is retained that is TVD for scalar equations~\cite{lev84}. LeVeque originally formulated his scheme as follows~\cite{lev85}:
\begin{subequations}\label{eq:levorig}
\begin{equation}\label{eq:proj}
  U_j^{n+1}=\frac{1}{\Delta x}\int_{x_{j-1/2}}^{x_{j+1/2}}u^*(x,t^{n+1})\dif x,
\end{equation}
where
\begin{equation}
  u^*(x,t)=u(x,t^n)+\sum_{\ell=-\infty}^{\infty}\left(\hat{u}_{\ell-1/2}(x,t)-\hat{u}_{\ell-1/2}(x,t^n)\right).
\end{equation}
\end{subequations}
 A main result of this paper is a closed form expression of LeVeque's method for scalar equations.
\begin{lem}
  For scalar equations, LeVeque's LTS-Godunov scheme can be written in closed form as follows:
\begin{equation}\label{eq:closedgod}
  U_j^{n+1}=\sum_{i=-\infty}^{\infty}\left(\frac{\Delta t}{\Delta x}\mathscr{M}_{j+1/2-i}\left(f(u)-(i-1)\frac{\Delta x}{\Delta t}u\right)-\frac{\Delta t}{\Delta x}\mathscr{M}_{j+1/2-i}\left(f(u)-i\frac{\Delta x}{\Delta t}u\right)-U_i^n\right),
\end{equation}
where the function $\mathscr{M}$ is defined as 
\begin{equation}
  \mathscr{M}_{j+1/2}(w(u))=\begin{cases} \min\limits_{u\in\mathcal{R}_{j+1/2}} w(u) & \text{if}\quad U_j<U_{j+1} \\ \max\limits_{u\in\mathcal{R}_{j+1/2}} w(u) & \text{if}\quad U_j\geq U_{j+1} \end{cases},
\end{equation} 
with
\begin{equation}
  \mathcal{R}_{j+1/2}=\left[\min(U_j,U_{j+1}),\max(U_j,U_{j+1})\right].  
\end{equation}
\end{lem}
\begin{proof}
Using \eqref{eq:piececonst}, we may rewrite \eqref{eq:levorig} as
\begin{multline}
  U_j^{n+1}=U_j^n+\frac{1}{\Delta x}\int_{x_{j-1/2}}^{x_{j+1/2}}\sum_{\ell=-\infty}^{\infty}\left(\hat{u}_{\ell-1/2}(x,t^{n+1})-\hat{u}_{\ell-1/2}(x,t^n)\right)\dif x \\
   =\frac{1}{\Delta x}\sum_{\ell=-\infty}^{\infty}\int_{x_{j-1/2}}^{x_{j+1/2}}\hat{u}_{\ell-1/2}(x,t^{n+1})\dif x-\sum_{\ell=-\infty}^{\infty}U_\ell^n.
\end{multline}
By changing the integration variable to $\zeta$, we obtain
\begin{equation}\label{eq:withzeta}
  U_j^{n+1}=\frac{\Delta t}{\Delta x}\sum_{i=-\infty}^{\infty}\int_{(i-1)\frac{\Delta x}{\Delta t}}^{i\frac{\Delta x}{\Delta t}}\hat{u}_{j+1/2-i}(\zeta)\dif\zeta-\sum_{i=-\infty}^{\infty}U_i^n.
\end{equation}
From \eqref{eq:oshu} we then directly obtain
\begin{equation}
    \int_{(i-1)\frac{\Delta x}{\Delta t}}^{i\frac{\Delta x}{\Delta t}}\hat{u}_{j+1/2-i}(\zeta)\dif\zeta=\mathscr{M}_{j+1/2-i}\left(f(u)-(i-1)\frac{\Delta x}{\Delta t}u\right)-\mathscr{M}_{j+1/2-i}\left(f(u)-i\frac{\Delta x}{\Delta t}u\right),
\end{equation}
completing the proof.
\end{proof}
\begin{lem}\label{lem:godfds}
  For scalar equations, LeVeque's LTS-Godunov scheme can be written in the flux-difference splitting formulation \eqref{eq:fdsform} with fluctuations
\begin{subequations}\label{eq:godfds}
\begin{gather}
  \left[\mathcal{A}^{i+}\Delta\right]_{j-1/2-i}=\mathscr{M}_{j-1/2-i}\left(f(u)-(i+1)\frac{\Delta x}{\Delta t}u\right)-\mathscr{M}_{j-1/2-i}\left(f(u)-i\frac{\Delta x}{\Delta t}u\right)+\frac{\Delta x}{\Delta t}U_{j-i} \\
  \left[\mathcal{A}^{i-}\Delta\right]_{j+1/2+i}=\mathscr{M}_{j+1/2+i}\left(f(u)+i\frac{\Delta x}{\Delta t}u\right)-\mathscr{M}_{j+1/2+i}\left(f(u)+(i+1)\frac{\Delta x}{\Delta t}u\right)+\frac{\Delta x}{\Delta t}U_{j+i},
\end{gather}
\end{subequations}
where
\begin{equation}\label{eq:Azero}
  \mathcal{A}^{i\pm}=0
\end{equation}
for
\begin{equation}\label{eq:largei}
  i>\frac{\Delta t}{\Delta x}\max\limits_u\left|f'(u)\right|.
\end{equation}
\end{lem}
\begin{proof}
The fluctuations \eqref{eq:godfds} are recovered by rearranging \eqref{eq:closedgod} in the context of \eqref{eq:fdsform}. When \eqref{eq:largei} holds, we have that
\begin{equation}
  f(U_{\mathrm{a}})-f(U_{\mathrm{b}})\leq\max\limits_u\left|f'(u)\right|\left(U_{\mathrm{a}}-U_{\mathrm{b}}\right)\leq i\frac{\Delta x}{\Delta t}\left(U_{\mathrm{a}}-U_{\mathrm{b}}\right)\leq (i+1)\frac{\Delta x}{\Delta t}\left(U_{\mathrm{a}}-U_{\mathrm{b}}\right)
\end{equation}
for all $U_{\mathrm{b}} \leq U_{\mathrm{a}}$. Then \eqref{eq:Azero} follows from \eqref{eq:godfds}.
\end{proof}
Hence, defining the global Courant number as
\begin{equation}\label{eq:globcou}
  \bar{C}=\frac{\Delta t}{\Delta x}\max\limits_{x,t}\left|f'(u(x,t))\right|,  
\end{equation}
we observe that LTS-Godunov is a $(2k+1)$ point scheme where
\begin{equation}
  k=\lceil\bar{C}\rceil.
\end{equation}
\begin{prop}
  For scalar equations, the numerical flux of LeVeque's LTS-Godunov scheme can be written in closed form as follows:
\begin{multline}\label{eq:godflux}
  F_{j+1/2}=\mathscr{M}_{j+1/2}\left(f(u)\right)-\sum_{i=1}^{\infty}\left[\left(f-i\frac{\Delta x}{\Delta t}u\right)_{j+1-i}-\mathscr{M}_{j+1/2-i}\left(f(u)-i\frac{\Delta x}{\Delta t}u\right)\right] \\
-\sum_{i=1}^{\infty}\left[\left(f+i\frac{\Delta x}{\Delta t}u\right)_{j+i}-\mathscr{M}_{j+1/2+i}\left(f(u)+i\frac{\Delta x}{\Delta t}u\right)\right].
\end{multline}
\end{prop}
\begin{proof}
  Observe that \eqref{eq:godflux} is consistent in the sense of \eqref{eq:consistent} and equivalent to \eqref{eq:godfds} through the transformation
\begin{equation}
  F_{j+1/2}-F_{j-1/2}=\frac{\Delta x}{\Delta t}\sum_{i=0}^{\infty}\left(\mathcal{A}^{i+}_{j-1/2-i}\Delta_{j-1/2-i}+\mathcal{A}^{i-}_{j+1/2+i}\Delta_{j+1/2+i}\right). 
\end{equation}
\end{proof}
\begin{prop}\label{prop:Qgod}
For scalar equations, the LTS-Godunov scheme can be expressed in the viscosity formulation \eqref{eq:viscform} with coefficients
\begin{subequations}\label{eq:Qgod}
\begin{align}
  Q^0_{j+1/2}&=\frac{\Delta t}{\Delta x}\frac{f_j+f_{j+1}-2\mathscr{M}_{j+1/2}\left(f(u)\right)}{U_{j+1}-U_j}, \\
  Q^{i+}_{j+1/2}&=\frac{\left(\frac{\Delta t}{\Delta x}f+iu\right)_{j}-\mathscr{M}_{j+1/2}\left(\frac{\Delta t}{\Delta x}f(u)+iu\right)}{U_{j+1}-U_j}, \\
  Q^{i-}_{j+1/2}&=\frac{\left(\frac{\Delta t}{\Delta x}f-iu\right)_{j+1}-\mathscr{M}_{j+1/2}\left(\frac{\Delta t}{\Delta x}f(u)-iu\right)}{U_{j+1}-U_j}.
\end{align}
\end{subequations}
\end{prop}
\begin{proof}
  Substitute into \eqref{eq:viscform} to recover \eqref{eq:godflux}. Alternatively, use the transformation~\eqref{eq:A2Q}.
\end{proof}
Note that LTS-Godunov is a {\em local} scheme in the sense of Definition~\ref{def:local}, and that
\begin{equation}
  Q^{\pm i}=0\quad\text{for}\quad i>\frac{\Delta t}{\Delta x}\max\limits_u\left|f'(u)\right|
\end{equation}
follows from the same reasoning as applied for Lemma~\ref{lem:godfds}.

\section{TVD Analysis}\label{sec:TVD}

LeVeque~\cite{lev84} proved that the LTS-Godunov scheme presented above satisfies the TVD property \eqref{eq:defTVD} and hence converges to weak solutions for scalar conservation laws. A highly classical result states the TVD condition for 3-point schemes:
\begin{lem}\label{lem:harten}
  A 3-point conservative scheme in the form \eqref{eq:3visc} is unconditionally TVD if and only if
\begin{equation}
   \left|C_{j+1/2}^n\right|\leq Q^n_{j+1/2}\leq 1
\end{equation}
for all $j$.
\end{lem}
This condition was proved to be sufficient by Harten~\cite{har83} and necessary by Tadmor~\cite{tad84a}. By the results of Osher~\cite{osh84}, it can be seen that the 3-point classical Godunov scheme satisfies this condition for Courant numbers less than one.

We will now prove analogous results for LTS methods.

\begin{lem}\label{lem:TVDcondfds}
 A local multi-point conservative scheme in the form \eqref{eq:fdsform} is unconditionally TVD if and only if
\begin{subequations}\label{eq:TVDcondfds}
\begin{align}
  \frac{\Delta x}{\Delta t}-\mathcal{A}^{0+}_{j+1/2}+\mathcal{A}^{0-}_{j+1/2}&\geq 0, \\
  \mathcal{A}^{i+}_{j+1/2}-\mathcal{A}^{(i+1)+}_{j+1/2}&\geq 0\quad\forall i \geq 0, \\
  \mathcal{A}^{(i+1)-}_{j+1/2}-\mathcal{A}^{i-}_{j+1/2}&\geq 0\quad\forall i \geq 0,
\end{align}
\end{subequations}
for all $j$.
\end{lem}
\begin{proof}
The proof is a natural generalization of the 3-point proof of Tadmor~\cite{tad84a}. Differencing \eqref{eq:fdsform} we obtain
\begin{multline}
  \frac{\Delta x}{\Delta t}\Delta_{j+1/2}^{n+1}=\left(\frac{\Delta x}{\Delta t}-\mathcal{A}^{0+}+\mathcal{A}^{0-}\right)_{j+1/2}\Delta_{j+1/2}+\sum_{i=1}^{\infty}\left(\mathcal{A}^{i-}-\mathcal{A}^{(i-1)-}\right)_{j+1/2+i}\Delta_{j+1/2+i} \\ +\sum_{i=1}^{\infty}\left(\mathcal{A}^{(i-1)+}-\mathcal{A}^{i+}\right)_{j+1/2-i}\Delta_{j+1/2-i}.
\end{multline}
Consider now the initial condition~\eqref{eq:initjump}. We obtain
\begin{equation}
  \frac{\Delta x}{\Delta t}\Delta_{j+1/2}^{n+1}=\begin{cases} \left[\mathcal{A}^{i+}(U_{\mathrm{L}},U_{\mathrm{R}})-\mathcal{A}^{(i+1)+}(U_{\mathrm{L}},U_{\mathrm{R}})\right]\Delta_{p+1/2}^n & \text{for}\quad p=j+1+i\quad\forall i\geq 0 \\ \left[\frac{\Delta x}{\Delta t}-\mathcal{A}^{0+}(U_{\mathrm{L}},U_{\mathrm{R}})+\mathcal{A}^{0-}(U_{\mathrm{L}},U_{\mathrm{R}})\right]\Delta^n_{p+1/2} & \text{for}\quad p=j \\ \left[\mathcal{A}^{(i+1)-}(U_{\mathrm{L}},U_{\mathrm{R}})-\mathcal{A}^{i-}(U_{\mathrm{L}},U_{\mathrm{R}})\right]\Delta_{p+1/2}^n & \text{for}\quad p=j-1-i\quad\forall i\geq 0 \end{cases}.  
\end{equation}
Now, in order for the scheme to be {\em monotonicity preserving}, $\Delta^{n+1}_{j+1/2}$ must have the same sign as $\Delta^n_{p+1/2}$ for all $j$. Then the inequalities \eqref{eq:TVDcondfds} must hold for {\em arbitrary} $(U_{\mathrm{L}},U_{\mathrm{R}})$. Jameson and Lax~\cite{jam86,jam87} and Osher and Chakravarthy~\cite{osh86} proved that \eqref{eq:TVDcondfds} are {\em sufficient} TVD conditions. As a TVD scheme is also monotonicity preserving~\cite{har83}, the result follows. 
\end{proof}
\begin{prop}
  A local multi-point conservative scheme in the form \eqref{eq:viscform} is unconditionally TVD if and only if
\begin{subequations}\label{eq:TVDcond}
\begin{align}
  1-Q^0_{j+1/2}+Q^{1-}_{j+1/2}+Q^{1+}_{j+1/2}&\geq 0,\label{eq:Qa} \\
  Q^0_{j+1/2}-4Q^{1\pm}_{j+1/2}+2Q^{2\pm}_{j+1/2}\mp C_{j+1/2}&\geq 0,\label{eq:Qrb} \\
  Q^{i\pm}_{j+1/2}-2Q^{(i+1)\pm}_{j+1/2}+Q^{(i+2)\pm}_{j+1/2} &\geq 0\quad\forall i\geq 1 \label{eq:Qri} 
\end{align}
\end{subequations}
for all $j$.
\end{prop}
\begin{proof}
The result follows from Lemma~\ref{lem:TVDcondfds} and the transformation \eqref{eq:Q2A}. 
\end{proof}
\begin{lem}\label{lem:LFineq}
  The viscosity coefficients of a local $(2k+1)$ point TVD scheme satisfy the inequalities
\begin{subequations}
\begin{gather}
  Q^0\leq k,\label{eq:Q0ineq} \\
  Q^{i\pm}\leq \frac{k-i}{2k}\left(\mp C+k\right)\label{eq:Qiineq}\quad\mathrm{for}\quad i=1,\ldots,k-1.
\end{gather}
\end{subequations}
These are either all strict inequalities or all equalities.
\end{lem}
\begin{proof}
Using \eqref{eq:zeroconv}, we get from induction on \eqref{eq:Qri}  that
\begin{equation}\label{eq:indres}
  Q^{i\pm}\leq\frac{k-i}{k+1-i}Q^{(i-1)\pm}\quad\text{for}\quad\text{for}\quad i\geq 2.   
\end{equation}
Using this, we can eliminate $Q^{2\pm}$ from \eqref{eq:Qrb} to obtain from \eqref{eq:indres}:
\begin{equation}
  \frac{2k}{k-i}Q^{i\pm}\leq\mp C+Q^0\quad\text{for}\quad i=1,\ldots,k-1.\label{eq:Q1temp}
\end{equation}
Using this in \eqref{eq:Qa} we recover \eqref{eq:Q0ineq}. Then \eqref{eq:Q1temp} gives us \eqref{eq:Qiineq} for $i=1,\ldots,k-1$.

Assume that equality holds in \eqref{eq:Q0ineq}. Then \eqref{eq:Qa} and \eqref{eq:Q1temp} give equality for $i=1$ in \eqref{eq:Qiineq}. Similarly, \eqref{eq:Qrb} and \eqref{eq:Q1temp} give equality for $i=2$ in \eqref{eq:Qiineq}. Repeated use of \eqref{eq:Qri} and \eqref{eq:Qiineq} impose the remaining equalities. Hence equality in \eqref{eq:Q0ineq} implies equality in \eqref{eq:Qiineq} for all $i=1,\ldots,k-1$.

Assume that equality holds for some $Q^{i\pm}$ in \eqref{eq:Qiineq}. Then from \eqref{eq:Q1temp} we obtain
\begin{equation}
  k\leq Q^0,
\end{equation}
which imposes equality in \eqref{eq:Q0ineq} and hence in \eqref{eq:Qiineq} for all $i=1,\ldots,k-1$.
\end{proof}
Note in particular that it follows from \eqref{eq:indres} and \eqref{eq:zeroconv} that
\begin{equation}\label{eq:posQ}
  Q^{i\pm}\geq 0\quad\text{for}\quad i\geq 1.  
\end{equation}

\begin{lem}\label{lem:roelim}

The viscosity coefficients of a local $(2k+1)$ point TVD scheme satisfy the inequalities
\begin{subequations}
\begin{gather}
 Q^0\geq |C|,\label{eq:R1} \\
 Q^{i\pm} \geq\max\left(0,\mp C-i\right).\label{eq:R2}
\end{gather}
\end{subequations}
\end{lem}
\begin{proof}
From \eqref{eq:Qri} and \eqref{eq:Q1temp} with $i=2$ we obtain
\begin{equation}
  Q^0\geq\frac{2k}{k-1}Q^{1\pm}\pm C,
\end{equation}
which by \eqref{eq:posQ} proves~\eqref{eq:R1}. Setting $i=1$ in \eqref{eq:Qiineq} and combining with \eqref{eq:Qa} we obtain
\begin{equation}\label{eq:S532}
  Q^{1\pm}\geq Q^0-1-Q^{1\mp}\geq Q^0-1-\frac{k-1}{2k}\left(Q^0\pm C\right)=\frac{k+1}{2k}Q^0-1\mp\frac{k-1}{2k}C.
\end{equation}
Furthermore, repeated application of \eqref{eq:Qri} gives
\begin{equation}
  Q^{(i+1)\pm}-Q^{i\pm}\geq Q^{2\pm}-Q^{1\pm}\quad\text{for}\quad i\geq 1.
\end{equation}
Combining this with \eqref{eq:Qrb} and \eqref{eq:S532} we obtain
\begin{equation}
  Q^{(i+1)\pm}\geq Q^{i\pm}+Q^{1\pm}-\frac{Q^0\mp C}{2}\geq Q^{i\pm}+\frac{1}{2k}\left(Q^0\pm C\right)-1\quad\text{for}\quad i\geq 1.\label{eq:S533}
\end{equation}
From \eqref{eq:R1} and \eqref{eq:S532} we now get
\begin{equation}
  Q^{1\pm}\geq\frac{k-1}{2k}\left(|C|\mp C\right)-1+\frac{|C|}{k}\geq\mp C-1.
\end{equation}
Assume now that $Q^{i\pm}\geq \mp C-i$. Then we get from \eqref{eq:R1} and \eqref{eq:S533} that
\begin{equation}
  Q^{(i+1)\pm}\geq Q^{i\pm}+\frac{1}{2k}\left(|C|\pm C\right)-1\geq\mp C-(i+1)+\frac{1}{2k}\left(|C|\pm C\right)\geq\mp C-(i+1),
\end{equation}
proving
\begin{equation}
  Q^{i\pm}\geq \mp C-i\quad\text{for}\quad i\geq 1
\end{equation}
by induction. Now \eqref{eq:R2} follows from \eqref{eq:posQ}.
\end{proof}
Observe that \eqref{eq:Q0ineq} and \eqref{eq:R1} provide the necessary CFL condition for $(2k+1)$-point local schemes to be TVD:
\begin{equation}\label{eq:genTVD}
  \max\limits_j\left|C_{j+1/2}\right|\leq k.
\end{equation}
We will now proceed to show that our upper and lower bounds on the numerical viscosity coefficients define the natural Large Time Step extensions of the Lax-Friedrichs and the Roe scheme respectively.

\subsection{The Large Time Step Lax-Friedrichs scheme}

According to Lemma~\ref{lem:harten}, the largest possible viscosity coefficient for 3-point TVD schemes is 
\begin{equation}
  Q_{j+1/2}^n=1.
\end{equation}
Substituting this into \eqref{eq:3visc}, we recover the classical {\em Lax-Friedrichs} scheme:
\begin{equation}
   U_j^{n+1}-\frac{1}{2}\left(U_{j-1}^n+U_{j+1}^n\right)+\frac{\Delta t}{2\Delta x}\left(f(U_{j+1}^n)-f(U_{j-1}^n)\right)=0.
\end{equation}
This has a natural $(2k+1)$ point generalization.
\begin{defn}
The $(2k+1)$ point numerical scheme given by
\begin{equation}\label{eq:defLTSLxF}
  U_j^{n+1}-\frac{1}{2}\left(U_{j-k}^n+U_{j+k}^n\right)+\frac{\Delta t}{2k\Delta x}\left(f(U_{j+k}^n)-f(U_{j-k}^n)\right)=0
\end{equation}
will be denoted as the {\bf LTS-LxF} scheme.
\end{defn}

\begin{lem}\label{lem:QLxF}
The LTS-LxF scheme can be written as a local conservative $(2k+1)$ point scheme with viscosity coefficients
\begin{subequations}\label{eq:LFQ}
\begin{gather}
  Q^0=k,\label{eq:LF1} \\
  Q^{i\pm}=\frac{k-i}{2k}\left(\mp C+k\right)\quad\mathrm{for}\quad i=1,\ldots,k-1.\label{eq:LF2}
\end{gather}
\end{subequations}
\end{lem}
\begin{proof}
  Substitute \eqref{eq:LFQ} into \eqref{eq:fdsform} and then into \eqref{eq:FV} to recover \eqref{eq:defLTSLxF}.
\end{proof}

\begin{prop}
  The local $(2k+1)$-point TVD scheme with the largest numerical viscosity \eqref{eq:defD} is the LTS-LxF scheme.
\end{prop}
\begin{proof}
  Observe that under the necessary condition \eqref{eq:genTVD}, the coefficients \eqref{eq:LFQ} satisfy the TVD conditions \eqref{eq:TVDcond}. To increase $D(u)$ in \eqref{eq:defD}, we would need to increase at least one of the $Q$ of \eqref{eq:LFQ}. According to Lemma~\ref{lem:LFineq}, this is impossible without violating the TVD property.
\end{proof}

\subsection{The Large Time Step Roe scheme}

Recall that the classical Godunov's method consists in projecting the exact solution \eqref{eq:oshu} of non-interacting Riemann problems back to the computational grid. One may also envisage the method used in the context of {\em approximate Riemann solvers} where some approximate solution is used in place of \eqref{eq:oshu}. One such method, commonly attributed to Roe~\cite{roe81}, consists of replacing the original nonlinear problem \eqref{eq:scalarcons} with a local linear problem
\begin{equation}\label{eq:Roelin}
  u_t+\hat{A}_{j+1/2}u_x=0  
\end{equation}
at each cell interface. Then, if $\hat{A}$ satisfies the condition
\begin{equation}\label{eq:condroe}
  \hat{A}_{j+1/2}\cdot(U_{j+1}-U_j)=f(U_{j+1})-f(U_j),
\end{equation}
setting $\hat{u}_{j+1/2}(x,t)$ to be the solution to the Riemann problem \eqref{eq:ciriemann} for \eqref{eq:Roelin} yields a conservative approximation to Godunov's method. For scalar equations, this is also known as Murman's method~\cite{lev02,mur74}.

Analogous to the LTS-Godunov scheme, a Large Time Step generalization of the Roe scheme was proposed by LeVeque~\cite{lev84}.
\begin{defn}
Consider the LTS-Godunov scheme \eqref{eq:levorig} where $\hat{u}_{j+1/2}(x,t)$ is interpreted as the solution to the initial value problem \eqref{eq:ciriemann} for \eqref{eq:Roelin}, where the matrix $\hat{A}$ satisfies the conditions of Roe~\cite{roe81}. Such a scheme will be denoted as a {\bf LTS-Roe} scheme.
\end{defn}
\begin{prop}\label{prop:Roesys}
  Given a Roe matrix
\begin{equation}\label{eq:RoeEig}
  \hat{A}_{j+1/2}=\left(\hat{R}{\hat{\Lambda}}\hat{R}^{-1}\right)_{j+1/2}\quad\forall j,
\end{equation}
where $\hat{\Lambda}$ is the diagonal matrix of eigenvalues, the LTS-Roe scheme for systems can be written in the flux-difference splitting form \eqref{eq:fdsform} with coefficients
\begin{equation}
  \mathcal{A}^{i\pm}_{j+1/2}=\hat{A}_{j+1/2}^{i\pm}=\left(\hat{R}\hat{\Lambda}^{i\pm}\hat{R}^{-1}\right)_{j+1/2},
\end{equation}
where for each eigenvalue we define
\begin{subequations}
\begin{gather}
  \hat{\lambda}^{i+}=\max\left(0,\min\left(\hat{\lambda}-i\frac{\Delta x}{\Delta t},\frac{\Delta x}{\Delta t}\right)\right),\\
  \hat{\lambda}^{i-}=\min\left(0,\max\left(\hat{\lambda}+i\frac{\Delta x}{\Delta t},-\frac{\Delta x}{\Delta t}\right)\right).
\end{gather}
\end{subequations}
\end{prop}
\begin{proof}
  For linear systems, we have that
\begin{equation}\label{eq:heavisol}
  \hat{u}_{j+1/2}(\zeta)=U_j+\hat{R}\hat{H}(\zeta I-\hat{\Lambda})\hat{R}^{-1}\left(U_{j+1}-U_j\right)=U_{j+1}-\hat{R}\hat{H}(\hat{\Lambda}-\zeta I)\hat{R}^{-1}\left(U_{j+1}-U_j\right),
\end{equation}
where we use the convention that $\hat{H}(\cdot)$ is the diagonal matrix obtained from applying the Heaviside function to each diagonal element of its argument. Then we obtain
\begin{equation}
  \int_{(i-1)\frac{\Delta x}{\Delta t}}^{i\frac{\Delta x}{\Delta t}}\hat{u}_{j+1/2-i}(\zeta)\dif\zeta=\begin{cases} \frac{\Delta x}{\Delta t}U_{j-i}-\hat{A}_{j+1/2-i}^{(-i)-}\left(U_{j+1-i}-U_{j-i}\right) & \text{for}\quad i\leq 0 \\ \frac{\Delta x}{\Delta t}U_{j+1-i}-\hat{A}_{j+1/2-i}^{(i-1)+}\left(U_{j+1-i}-U_{j-i}\right) & \text{for}\quad i\geq 1, \end{cases}
\end{equation}
which inserted into \eqref{eq:withzeta} yields
\begin{equation}
  U_j^{n+1}=U_j^n-\frac{\Delta t}{\Delta x}\sum_{i=0}^{\infty}\left(\hat{A}_{j-1/2-i}^{i+}\Delta_{j-1/2-i}+\hat{A}^{i-}_{j+1/2+i}\Delta_{j+1/2+i}\right).
\end{equation}
\end{proof}

\begin{lem}\label{lem:scalarRoe}
For scalar equations, LTS-Roe is uniquely given as
\begin{subequations}\label{eq:mcARoe}
\begin{gather}
  \mathcal{A}^{i+}=\frac{\Delta x}{\Delta t}\max\left(0,\min\left(C-i,1\right)\right),\\
  \mathcal{A}^{i-}=\frac{\Delta x}{\Delta t}\min\left(0,\max\left(C+i,-1\right)\right).
\end{gather}
\end{subequations}
in the flux-difference splitting formulation and
\begin{subequations}\label{eq:Qroe}
\begin{gather}
  Q^0=|C|, \\
  Q^{i\pm}=\max(0,\mp C-i)
\end{gather}
\end{subequations}
in the viscosity formulation.
\end{lem}
\begin{proof}
For scalar equations, the unique Roe matrix is
\begin{equation}
  \hat{A}_{j+1/2}=\frac{\Delta x}{\Delta t}C_{j+1/2}.  
\end{equation}
We then obtain \eqref{eq:Qroe} from the transformation \eqref{eq:A2Q} on \eqref{eq:mcARoe}.
\end{proof}
\begin{prop}
The local $(2k+1)$-point TVD scheme with the smallest numerical viscosity~\eqref{eq:defD} is the LTS-Roe scheme.
\end{prop}
\begin{proof}
  Observe that the coefficients \eqref{eq:Qroe} unconditionally satisfy the TVD conditions \eqref{eq:TVDcond} and that LTS-Roe reduces to a $(2k+1)$-point scheme under the CFL condition \eqref{eq:genTVD}. To decrease $D(u)$ in \eqref{eq:defD}, we would need to decrease at least one of the $Q$ of \eqref{eq:LFQ}. According to Lemma~\ref{lem:roelim}, this is impossible without violating the TVD property.
\end{proof}

The following summary of these observations is a main result of our paper.
\begin{theorem}\label{thm:theoQ}
  Let $\mathrm{S}$ be some local $(2k+1)$ TVD scheme, and let $Q_{\mathrm{S}},Q_{\mathrm{Roe}}$ and $Q_{\mathrm{LxF}}$ be the numerical viscosity coefficients of $\mathrm{S}$, LTS-Roe and the $(2k+1)$-point LTS-LxF scheme respectively. Then, unless $\mathrm{S}$ is precisely the LTS-LxF scheme, we have the sharp inequalities
\begin{equation}
  Q_{\mathrm{Roe}}^\alpha\leq Q_{\mathrm{S}}^\alpha< Q_{\mathrm{LxF}}^\alpha\quad\forall\alpha=0,i\pm.
\end{equation}
\end{theorem}
\begin{proof}
  The result follows from Lemmas \ref{lem:LFineq}, \ref{lem:roelim}, \ref{lem:QLxF} and \ref{lem:scalarRoe}.
\end{proof}
In particular, this implies that for a cell interface that satisfies
\begin{equation}
  \left|C_{j+1/2}\right|=k,  
\end{equation}
we have that
\begin{equation}
  Q^{\alpha}_{\mathrm{Roe}}=Q^{\alpha}_{\mathrm{LxF}}\quad\forall\alpha=0,i\pm.  
\end{equation}

\subsection{LTS-Roe vs LTS-Godunov}

For scalar equations, the classical Godunov and Roe schemes always coincide except in the presence of {\em transonic rarefactions}, i.e.\ when the Courant number \eqref{eq:contcou} changes sign across the cell interface~\cite{lev02}. In this section, we will investigate how this relationship extends to Large Time Steps. We start with a useful reformulation of the LTS-Roe scheme.
\begin{prop}\label{prop:godroe}
 For scalar equations, the numerical flux of the LTS-Roe scheme can be written as follows:
\begin{subequations}\label{eq:godroe}
\begin{multline}
  F_{j+1/2}=\hat{\mathscr{M}}_{j+1/2}\left(f(u)\right)-\sum_{i=1}^{\infty}\left[\left(f-i\frac{\Delta x}{\Delta t}u\right)_{j+1-i}-\hat{\mathscr{M}}_{j+1/2-i}\left(f(u)-i\frac{\Delta x}{\Delta t}u\right)\right] \\
-\sum_{i=1}^{\infty}\left[\left(f+i\frac{\Delta x}{\Delta t}u\right)_{j+i}-\hat{\mathscr{M}}_{j+1/2+i}\left(f(u)+i\frac{\Delta x}{\Delta t}u\right)\right],
\end{multline}
where the function $\mathscr{M}$ is defined as 
\begin{equation}
  \hat{\mathscr{M}}_{j+1/2}(w(u))=\begin{cases} \min\limits_{u\in\{U_j,U_{j+1}\}} w(u) & \text{if}\quad U_j<U_{j+1} \\ \max\limits_{u\in\{U_j,U_{j+1}\}} w(u) & \text{if}\quad U_j\geq U_{j+1} \end{cases}.
\end{equation} 
\end{subequations}
\end{prop}
\begin{proof}
This can be proved in two ways. First, simply apply the LTS-Godunov scheme directly to the linearized flux function $\hat{f}_{j+1/2}(u)=\hat{A}_{j+1/2}u$, noting that the extrema of a linear function must reside on the boundaries. Alternatively, verify that \eqref{eq:godroe} is equivalent to \eqref{eq:Qroe}.
\end{proof}
\begin{prop}
  The numerical viscosity coefficients of the LTS-Roe and LTS-Godunov scheme satisfy
\begin{equation}
  Q_{\mathrm{God}}^{\alpha}\geq Q_{\mathrm{Roe}}^{\alpha}\quad\forall \alpha=0,i\pm,
\end{equation}
where equality holds except possibly when the Courant number \eqref{eq:contcou} acquires an integer value across the cell interface. More precisely, for all $i\geq 0$ we have
\begin{equation}\label{eq:godroeeq}
  Q_{\mathrm{God}}^{i\pm}(U_{j},U_{j+1})=Q_{\mathrm{Roe}}^{i\pm}(U_{j},U_{j+1})
\end{equation}
as long as there is no $u\in\mathcal{R}_{j+1/2}$ for which
\begin{equation}\label{eq:gentrans}
  c(u)=\mp i,
\end{equation}
where $c$ is defined by \eqref{eq:contcou}.
\end{prop}
\begin{proof}
  From Propositions~\ref{prop:Qgod} and \ref{prop:godroe} we obtain 
\begin{subequations}
\begin{align}
  Q^{0}_{\mathrm{God}}-Q^{0}_{\mathrm{Roe}}&=2\frac{\min\limits_{u\in\{U_j,U_{j+1}\}}\left[\frac{\Delta x}{\Delta t}f(u)\cdot\sigma\right]-\min\limits_{u\in\mathcal{R}_{j+1/2}}\left[\frac{\Delta x}{\Delta t}f(u)\cdot\sigma\right]}{|U_{j+1}-U_j|}, \\
  Q^{i\pm}_{\mathrm{God}}-Q^{i\pm}_{\mathrm{Roe}}&=\frac{\min\limits_{u\in\{U_j,U_{j+1}\}}\left[\left(\frac{\Delta x}{\Delta t}f(u)\pm iu\right)\cdot\sigma\right]-\min\limits_{u\in\mathcal{R}_{j+1/2}}\left[\left(\frac{\Delta x}{\Delta t}f(u)\pm iu\right)\cdot\sigma\right]}{|U_{j+1}-U_j|},
\end{align}
\end{subequations}
where we define
\begin{equation}
 \sigma(U_j,U_{j+1})=\begin{cases} 1 & \mathrm{if}\quad U_{j+1}-U_j \geq 0\\ -1 & \mathrm{otherwise.} \end{cases}  
\end{equation}
For the equality \eqref{eq:godroeeq} to be violated for $Q^{\pm i}$, the expression $\frac{\Delta t}{\Delta x}f(u)\pm iu$ must have an extremum in the interior of $\mathcal{R}_{j+1/2}$. Then its derivative must vanish, which is precisely the condition \eqref{eq:gentrans}.
\end{proof}

\subsubsection{Entropy violations} Observe that the solution \eqref{eq:heavisol} to the linearized Riemann problem \eqref{eq:ciriemann} and \eqref{eq:Roelin} contains only discontinuities. In particular, the Roe solver will propagate any rarefaction wave as an expansion shock, and we are dependent on the numerical diffusion introduced by the projection step~\eqref{eq:proj} to recover the correct entropy solution.

Now \eqref{eq:gentrans} is in fact the condition for no projection error; if 
\begin{equation}\label{eq:badcou}
  C_{j+1/2}=\mp i,
\end{equation}
the discontinuity will simply be translated $i$ grid cells for each time step. This may also be understood in terms of the modified equation \eqref{eq:modQ} for the LTS-Roe scheme, we obtain
\begin{equation}\label{eq:Droe}
  D(u)=\left(\lceil|c|\rceil-|c|\right)\left(1+|c|-\lceil|c|\rceil\right),
\end{equation}
which vanishes whenever \eqref{eq:gentrans} is satisfied.

For the LTS-Roe scheme, entropy violations associated with \eqref{eq:gentrans} have been reported by several authors~\cite{lev82,lev84,lev85,mor12a,qia11,xu14}.
 To remedy this, all these authors follow the approach of LeVeque~\cite{lev85} in splitting each rarefaction wave into a sequence of entropy-violating shocks that will spread out to approximate the true rarefaction wave as time evolves. If the rarefaction wave is thus split at points satisfying
\begin{equation}
  c(u^{i\pm *})=\mp i,
\end{equation}
this approach will in fact precisely recover the true LTS-Godunov scheme.

Based on the above heuristics, we will now propose a simple alternative entropy fix. Entropy violations arise when a Courant number~\eqref{eq:badcou} exactly preserves an expansion shock. We propose to destabilize this situation by always varying the time step;  at each time level we replace the fixed global Courant number $\bar{C}$ as defined in \eqref{eq:globcou} with the modification
\begin{equation}\label{eq:randCFL}
   \hat{C}=\bar{C}+r,  
\end{equation}
where $r$ is sampled from the uniform random distribution $\mathcal{U}(-1/2,1/2)$. This ensures that the critical Courant numbers \eqref{eq:badcou} will change at each time step, and travelling expansion shocks will be smeared by the projection \eqref{eq:proj}. We observe that this fix seems always to work in practice, and examples will be provided in Section~\ref{sec:simulations}.

Note that the fix will not work for the classical $C_{j+1/2}=0$ transonic rarefaction, as this degenerate case becomes independent of the time step. We therefore propose to combine the random time step method with a classical entropy fix, for instance the fix of Harten~\cite{har83}, to be able to handle expansion shocks of arbitrary velocity.
\begin{defn}
The scheme defined by the viscosity coefficients
\begin{subequations}\label{eq:Qroestar}
\begin{gather}
  Q^0=\begin{cases} |C| & \mathrm{if}\quad |C| \geq \delta, \\ (C^2+\delta^2)/(2\delta) & \mathrm{if}\quad|C|\leq\delta \end{cases} \\
  Q^{i\pm}=\max(0,\mp C-i)
\end{gather}
\end{subequations}
for some $0<\delta<1$, subject to time stepping according to \eqref{eq:randCFL}, will be denoted as the {\bf LTS-Roe*} scheme.
\end{defn}

\section{A Family of LTS Schemes}\label{sec:systems}

Similarly to \eqref{eq:Droe}, we can calculate from \eqref{eq:modQ} the total numerical viscosity for LTS-LxF:
\begin{equation}\label{eq:DLxF}
  D(u)=k^2-c^2.
\end{equation}
Hence by \eqref{eq:Droe} and Theorem~\ref{thm:theoQ} we have the sharp inequality for the numerical viscosity of $(2k+1)$-point local TVD schemes:
\begin{equation}\label{eq:sharpD}.
  \left(\lceil|c|\rceil-|c|\right)\left(1+|c|-\lceil|c|\rceil\right)\leq D(u)\leq k^2-c^2.
\end{equation}
Now 3-point schemes \eqref{eq:3visc} are uniquely defined through their numerical viscosity, but as remarked by Tadmor~\cite{tad84}, this does not hold when $k>1$; different $(2k+1)$-point schemes~\eqref{eq:viscform} may give the same $D(u)$ for some given $u$.

We would however be interested in a simple constructive procedure that would give us a local TVD scheme with a target numerical viscosity satisfying~\eqref{eq:sharpD}. In particular, we aim for a parametrization $Q(\beta;U_{\mathrm{L}},U_{\mathrm{R}})$, $\beta\in[0,1]$ with the following properties:
\begin{itemize} 
  \item[P1:] $Q^\alpha(\beta;U_{\mathrm{L}},U_{\mathrm{R}})$ satisfy the TVD conditions \eqref{eq:TVDcond} unconditionally;
  \item[P2:] $D(\beta;u)$ is a strictly monotonically increasing function of $\beta$;
  \item[P3:] $D(0;u)=\left(\lceil|c|\rceil-|c|\right)\left(1+|c|-\lceil|c|\rceil\right)$;
  \item[P4:] $D(1;u)=k^2-c^2$.
\end{itemize}
As any convex combination of TVD fluxes is TVD, this can be easily achieved.
\begin{defn}
The $(2k+1)$-point scheme given by
\begin{equation}
  Q^\alpha(\beta;U_{\mathrm{L}},U_{\mathrm{R}})=\beta Q^{\alpha}_{\mathrm{LxF}}(U_{\mathrm{L}},U_{\mathrm{R}})+(1-\beta)Q^{\alpha}_{\mathrm{Roe}}(U_{\mathrm{L}},U_{\mathrm{R}})
\end{equation}
will be denoted as the LTS-RoeLxF($\beta$) scheme.
\end{defn}
Herein, P1-P4 will continue to hold if $\beta$ is chosen individually at each cell interface at each time step. 

This LTS-RoeLxF($\beta$) method is not unique in satisfying P1-P4, nor do we claim it is the optimal such scheme. Rather, our intent is to concretely demonstrate that the level of numerical diffusion may be controlled in the Large Time Step TVD setting. This may be of particular relevance when we consider the extension to systems of nonlinear equations, as nonlinear interactions between the characteristic fields may induce numerical instabilities~\cite{lev85}.

\subsection{Extension to systems}

As there is no analogue of the TVD property for general nonlinear systems of equations, we will follow the standard approach~\cite{har86,qia12} of applying the scalar framework to a field-by-field decomposition obtained through the Roe linearization.

Consider now the class of scalar local multi-point schemes where the flux-difference splitting coefficients can be expressed as
\begin{equation}\label{eq:scalara}
  \mathcal{A}^{i\pm}=a^{i\pm}(U_{\mathrm{L}},U_{\mathrm{R}})=a^{i\pm}\left(C(U_{\mathrm{L}},U_{\mathrm{R}})\right),
\end{equation}
where $C$ is the signed local cell interface Courant number \eqref{eq:defcou}. This class of schemes naturally arises in the TVD setting, as $C$ is the only cell interface parameter featured in the TVD conditions~\eqref{eq:TVDcond}. In particular, note that both LTS-Roe and LTS-LxF (but not LTS-Godunov) belong to this class of schemes. Then the LTS-Roe scheme for systems, described in Proposition~\ref{prop:Roesys}, may be naturally generalized as follows; given the Roe matrix~\eqref{eq:RoeEig}, construct the flux-difference splitting coefficients as
\begin{equation}\label{eq:Roesplit}
  \mathcal{A}^{i\pm}_{j+1/2}=\hat{A}_{j+1/2}^{i\pm}=\left(\hat{R}\hat{\Lambda}^{i\pm}\hat{R}^{-1}\right)_{j+1/2},
\end{equation}
where for each eigenvalue we define
\begin{equation}
  \hat{\lambda}^{i\pm}=a^{i\pm}(\hat{\lambda}),
\end{equation}
where $a^{\pm i}$ are the functions~\eqref{eq:scalara} defining the scalar version of the scheme.

Herein, we may apply different schemes \eqref{eq:scalara} to the different characteristic families of \eqref{eq:Roesplit}. Note that also for nonlinear systems, applying this procedure componentwise to the Lax-Friedrichs coefficients \eqref{eq:LFQ} reduces the scheme to exactly~\eqref{eq:defLTSLxF}.

\section{Numerical Examples}\label{sec:simulations}

We will now provide some numerical illustrations on the relationship between numerical viscosity and entropy violations.

\subsection{Burgers equation}

We consider the Burgers equation in one space dimension:
\begin{subequations}\label{eq:burgers}
\begin{gather}
  u_t+\left( \frac{1}{2} u^2\right)_x=0, \quad u \in \mathbb{R}, \\
  u(x,0)=u_0(x).
\end{gather}
\end{subequations} 
We use a grid of 800 computational cells covering $x\in[0,1]$, and a Courant number $\bar{C}=5$. Furthermore, the parameter $\delta=0.5$ is used for Harten's entropy fix \eqref{eq:Qroestar}.

\subsubsection{Square pulse}

We consider the square initial data conditions:
\begin{equation}
  u_0(x) = \begin{cases}  1 & \text{if} \quad 0.3 < x < 0.7 \\
                         0 & \text{otherwise.}
          \end{cases}
  \label{eq:square_initial_data}
\end{equation}
The solution after $t=0.2$ is plotted in Figure~\ref{fig:schemes_square} for various schemes.
\begin{figure}
  \begin{center}
\begin{subfigure}[b]{0.47\textwidth}
    \includegraphics[width=1.00\textwidth]{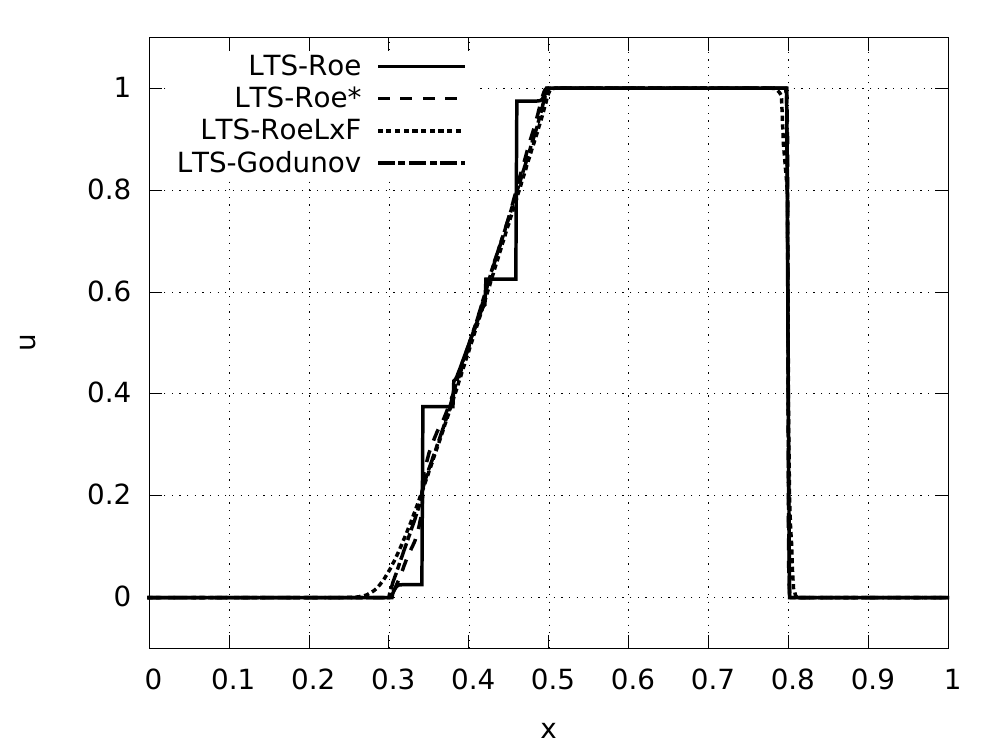}
    \caption{Comparion of various schemes.}
    \label{fig:square_comp}
\end{subfigure}
\begin{subfigure}[b]{0.47\textwidth}
    \includegraphics[width=1.00\textwidth]{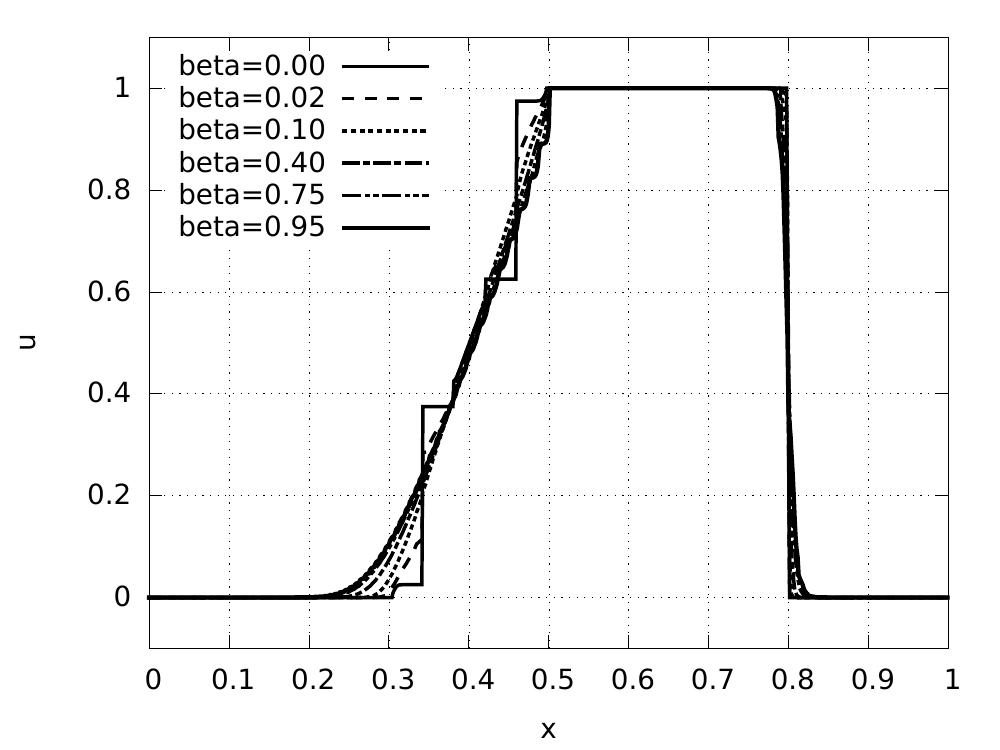}
   \caption{LTS-RoeLxF($\beta$) for different values of $\beta$.}
   \label{fig:square_beta}
\end{subfigure}
\caption{Square pulse initial data for Burgers equation.\label{fig:schemes_square}}
  \end{center}
\end{figure}
In Figure~\ref{fig:square_comp}, $\beta=0.2$ was used for the LTS-RoeLxF scheme. We observe that LTS-Roe splits the rarefaction wave into a number of entropy-violating expansion shocks, whereas the other schemes all produce a satisfactory resolution of the rarefaction wave.

In Figure~\ref{fig:square_beta}, various values for $\beta$ in the LTS-RoeLxF($\beta$) scheme are considered. We observe that increasing $\beta$ leads to the expected increase in numerical diffusion, and recovers the entropy-satisfying rarefaction wave for even small values of $\beta$.

\subsubsection{Transonic rarefaction wave}

We now wish to focus on the classical case of the transonic rarefaction, where entropy violations are expected also for low Courant numbers. We consider the initial conditions
\begin{equation}
  u_0(x) = \begin{cases} -1 & \text{if} \quad 0.25 < x \leq 0.5 \\
                          1 & \text{if} \quad 0.5 < x < 0.75 \\
                         0 & \text{otherwise.}
          \end{cases}
  \label{eq:initial_data}
\end{equation}
Results are plotted in Figure~\ref{fig:transonic}.
\begin{figure}[htbp]
  \begin{center}
\begin{subfigure}[b]{0.47\textwidth}
    \includegraphics[width=1.00\textwidth]{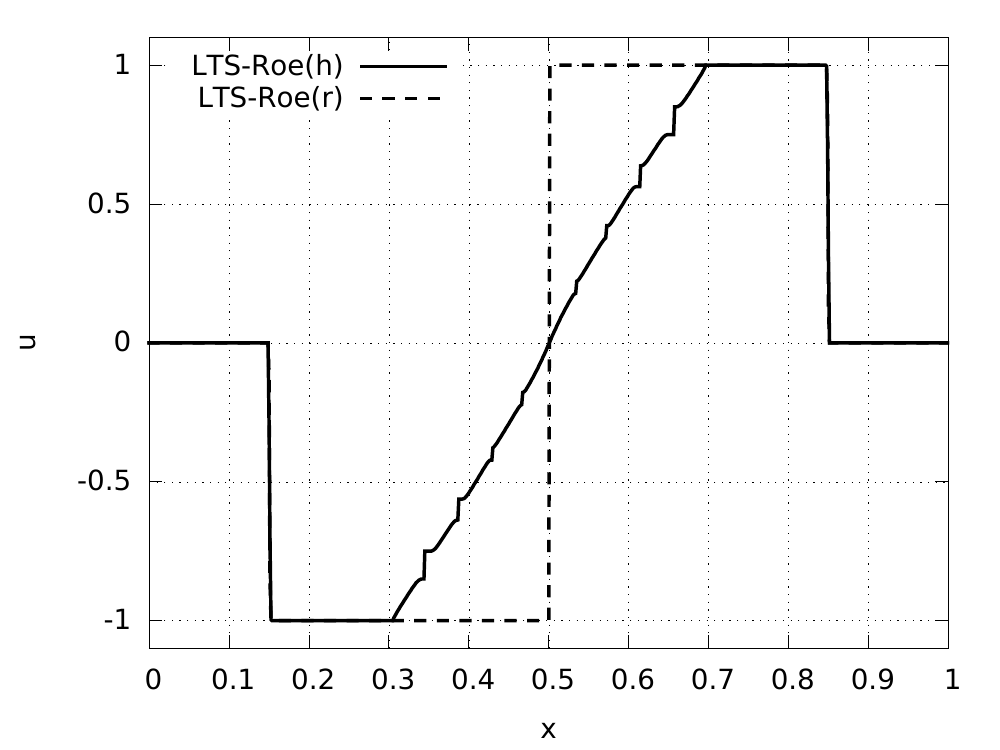}
    \caption{Harten's entropy fix vs random time step.\label{fig:hr}}
\end{subfigure}
\begin{subfigure}[b]{0.47\textwidth}
    \includegraphics[width=1.00\textwidth]{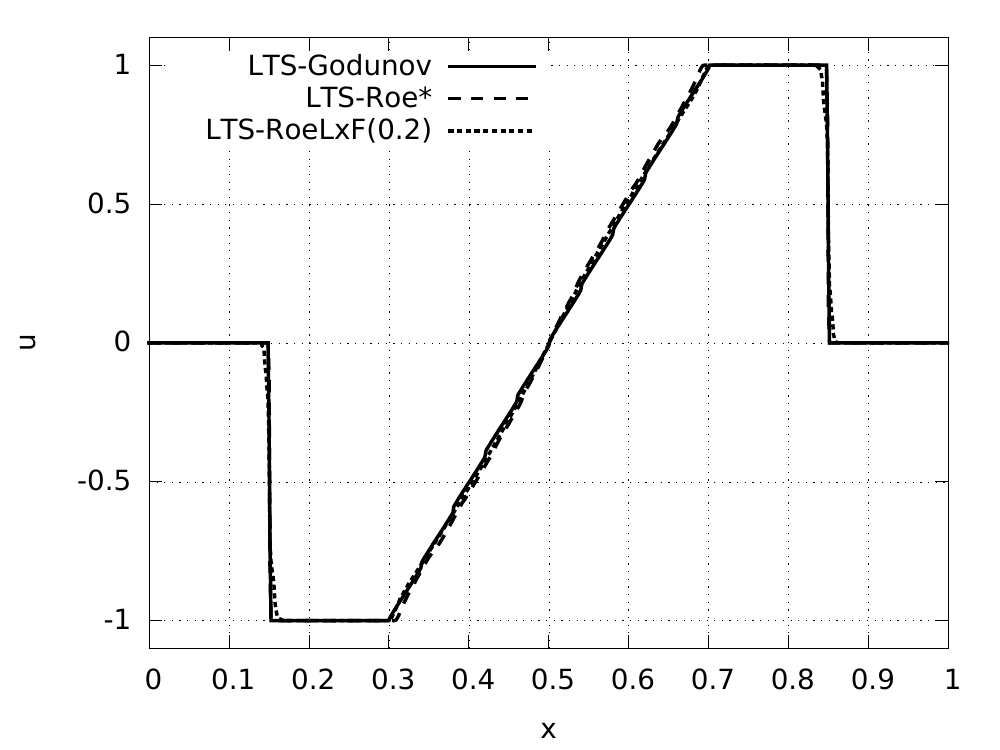}
    \caption{Entropy-satisfying solutions.\label{fig:entsat}}
\end{subfigure}
  \end{center}
  \caption{Transonic rarefaction for Burgers equation.\label{fig:transonic}}
\end{figure}
In Figure~\ref{fig:hr}, we consider the two entropy fixes of LTS-Roe* separately. Applying only the random time step \eqref{eq:randCFL}, denoted as LTS-Roe(r), will not modify the zero Courant number and the transonic rarefaction remains a stationary expansion shock. Applying only Harten's entropy fix \eqref{eq:Qroestar}, we recover a rarefaction interspersed with expansion shocks similar to Figure~\ref{fig:schemes_square}.

On the other hand, from Figure~\ref{fig:entsat} we see that both the full LTS-Roe* scheme and the LTS-RoeLxF(0.2) scheme resolve the rarefaction on par with the LTS-Godunov scheme.

\subsection{Sod shock tube problem}

We consider the classical Sod shock tube problem for the Euler equations of gas dynamics, described in detail in~\cite{qia11,sod78}. In Figure~\ref{fig:sod}, a comparison between LTS-Roe and LTS-Roe* after $t=0.25$ for different Courant numbers is presented.
\begin{figure}
  \begin{center}
\begin{subfigure}[b]{0.47\textwidth}
    \includegraphics[width=1.00\textwidth]{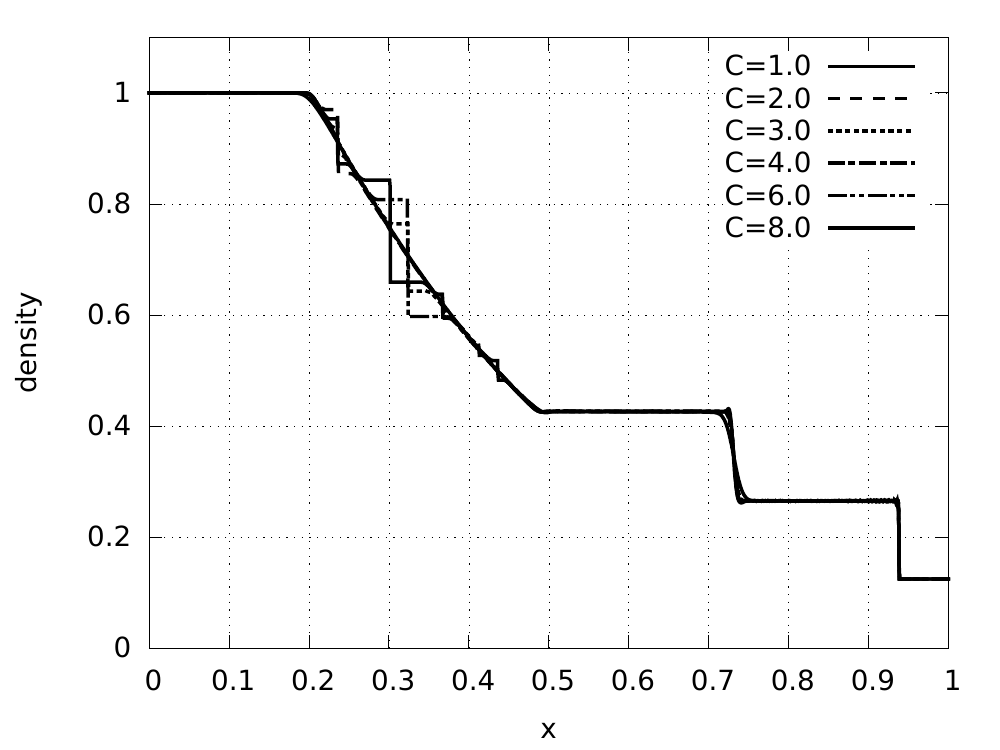}
    \caption{LTS-Roe\label{fig:roesod1}}
\end{subfigure}
\begin{subfigure}[b]{0.47\textwidth}
    \includegraphics[width=1.00\textwidth]{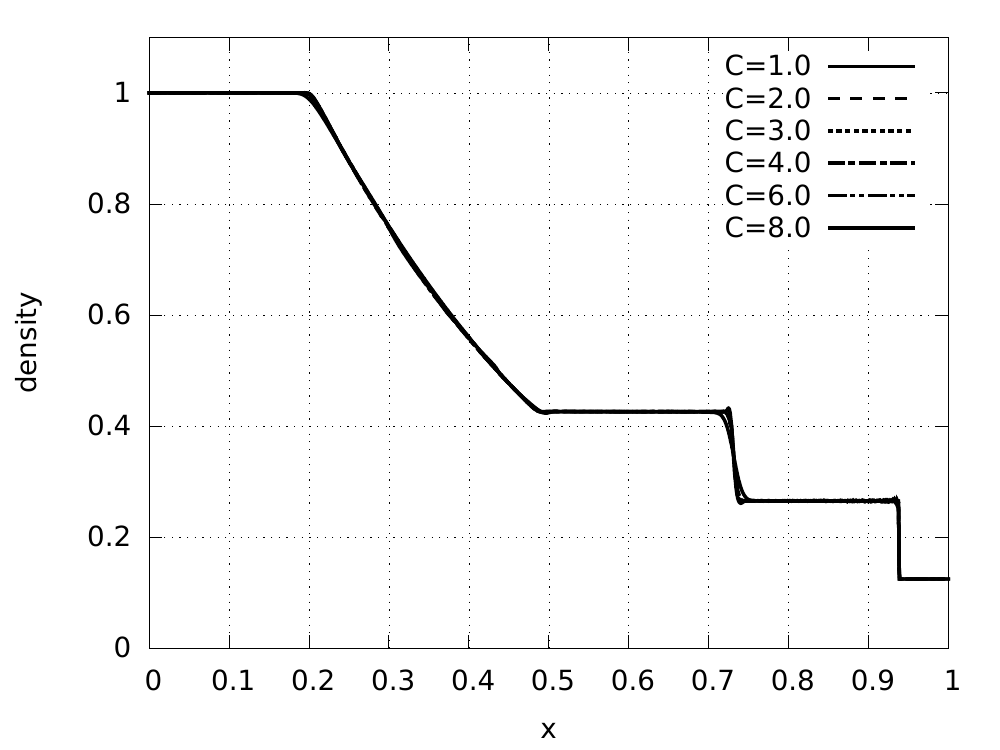} 
    \caption{LTS-Roe*\label{fig:roestarsod1}}
\end{subfigure}
  \end{center}
\caption{Density plot for the Sod shock tube problem, 1800 cells.\label{fig:sod}}
\end{figure}
For LTS-Roe, all Courant numbers of 2 and above here yield entropy-violating expansion shocks with velocities satisfying \eqref{eq:badcou}. The time stepping of the LTS-Roe* scheme once more removes these entropy violations.
\begin{figure}
  \begin{center}
\begin{subfigure}[b]{0.47\textwidth}
    \includegraphics[width=1.00\textwidth]{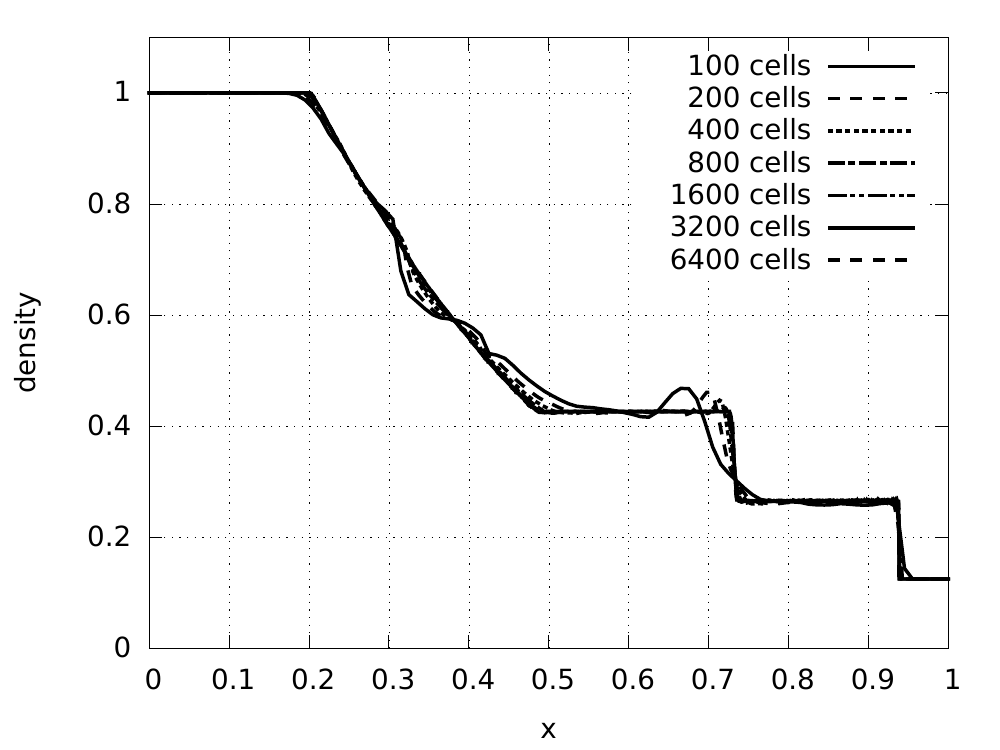}
  \caption{LTS-Roe*\label{fig:roestar}}
\end{subfigure}
\begin{subfigure}[b]{0.47\textwidth}
    \includegraphics[width=1.00\textwidth]{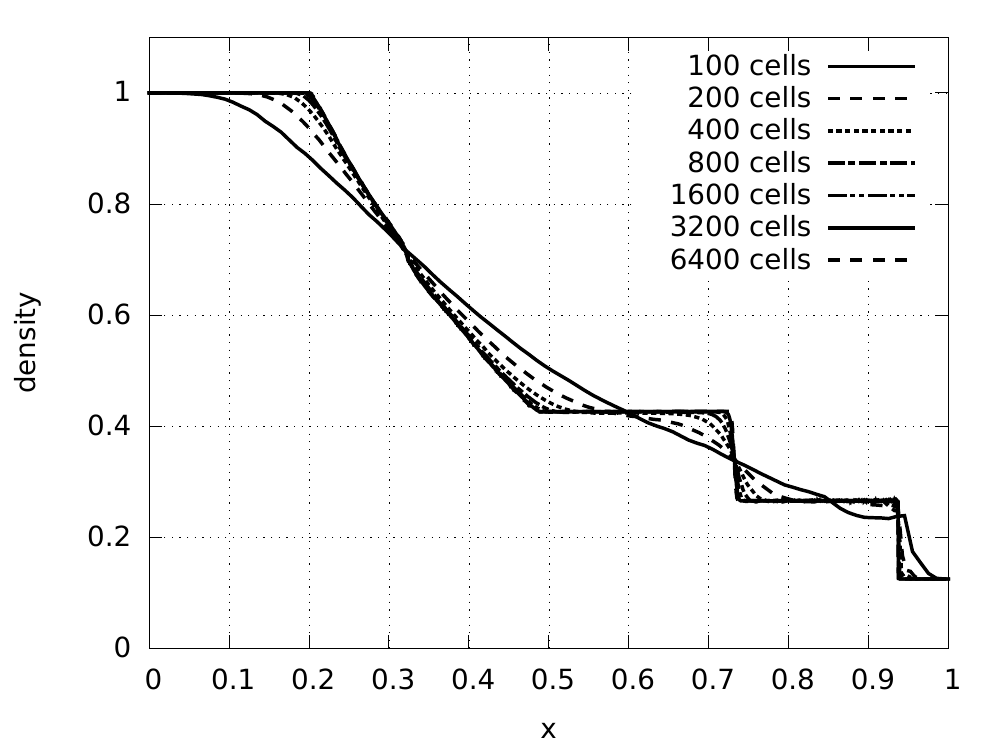} 
    \caption{LTS-RoeLxF, $\beta/\Delta x=30.0$\label{fig:beta}}
\end{subfigure}
  \end{center}
\caption{Density plot for the Sod shock tube problem, C=6.0.\label{fig:sod2}}
\end{figure}

However, Figure~\ref{fig:sod2} displays spurious oscillations associated with LTS-Roe* which reduce with grid refinement. For comparison, the same grid refinement study is performed on the LTS-RoeLxF($\beta$) scheme, were we adapt $\beta$ to the grid size acccording to
\begin{equation}
   \frac{\beta}{\Delta x}=30.0.
\end{equation}
This choice introduces a significant amount of numerical diffusion, eliminating both the spurious oscillations and entropy violations. The quality of the numerical solutions will start to suffer by a further increase of the Courant number.

\section{Summary}\label{sec:summary}

In this paper, we have expanded on classical works on multi-point TVD schemes, obtaining some original precise statements. In particular, by considering the restricted class of {\em local} multi-point schemes we were able to obtain the necessary and sufficient TVD conditions \eqref{eq:TVDcond}. The coefficients of these conditions are directly linked to the numerical viscosity through the modified equation analysis.

A main result of our paper is Proposition~\ref{prop:Qgod}, where we provide closed form expressions for these coefficients for LeVeque's LTS-Godunov scheme. By this, the LTS-Godunov scheme is explicitly placed in the numerical viscosity TVD framework.

A second main result is Theorem~\ref{thm:theoQ}, which extends to Large Time Steps Harten's result that the Lax-Friedrichs and Roe schemes provide the upper and lower limits for the coefficients of numerical viscosity in local TVD schemes.

We have provided some insights on entropy violations associated with the LTS-Roe scheme when travelling expansion shocks are aligned with the grid, illustrated  with numerical examples.

The framework may be naturally extended to nonlinear systems through the Roe linearization. Our hope is that the insights of this paper may facilitate further research on devising explicit schemes that appropriately combine accuracy and robustness under Large Time Steps. Herein, the general approach of Harten~\cite{har86} may be used to achieve high resolution.

\section{Acknowledgments}

The first and second authors were funded in part through the NORDICCS
Centre, performed under the Top-level Research Initiative CO2 Capture 
and Storage program, and Nordic Innovation. The authors acknowledge the 
following partners for their contributions: Statoil, Gassco, Norcem, 
Reykjavik Energy, and the Top-level Research Initiative (Project number 
11029).

The work of the third author was supported by the Research Council of Norway (234126/30) through the SIMCOFLOW project.

We are grateful to our colleagues Sigbj\o{}rn L\o{}land Bore, Stein Tore Johansen, Halvor Lund, Alexandre Morin, Bernhard M\"{u}ller, Marica Pelanti and Marin Prebeg for fruitful discussions.


\end{document}